\documentclass{article}

\usepackage{enumerate}

\usepackage{amsmath}
\usepackage{amstext,amssymb,amsfonts}
\usepackage{fullpage}
\usepackage{bm}

\usepackage{todonotes}

\newcounter{mynotes}
\usepackage{nameref}
\usepackage[pagebackref,colorlinks,linkcolor=blue,filecolor = blue,
citecolor = blue, urlcolor = blue,pdfstartview=FitH]{hyperref}
\usepackage[capitalize]{cleveref}
\usepackage{amsthm}
\usepackage{thmtools}

\declaretheorem[within=section]{theorem}
\declaretheorem[sibling=theorem]{corollary}

\declaretheorem[sibling=theorem]{lemma}
\declaretheorem[sibling=theorem]{claim}

\declaretheorem[sibling=theorem]{definition}
\declaretheorem[sibling=theorem]{Lemma+Definition}
\declaretheorem[sibling=theorem]{remark}

\declaretheorem[sibling=theorem]{conjecture}

\crefname{proposition}{Proposition}{Propositions}
\crefname{conjecture}{Conjecture}{Conjectures}
\crefname{claim}{Claim}{Claims}
\crefname{remark}{Remark}{Remarks}

\newcounter{termcounter}
\renewcommand{\thetermcounter}{\Alph{termcounter}}
\crefname{term}{term}{terms}
\creflabelformat{term}{\textup{(#2#1#3)}}

\makeatletter
\def\term{\@ifnextchar[\term@optarg\term@noarg}%]
\def\term@optarg[#1]#2{%
  \textup{(#1)}%
  \def\@currentlabel{#1}%
  \def\cref@currentlabel{[][2147483647][]#1}%
  \cref@label[term]{#2}}
\def\term@noarg#1{%
  \refstepcounter{termcounter}%
  \textup{(\thetermcounter)}%
  \cref@label[term]{#1}}
\makeatother
\newcommand{\ignore}[1]{}

\definecolor{DSred}{rgb}{1,0,0}

\renewcommand{\leq}{\leqslant}

\renewcommand{\ge}{\geqslant}
\renewcommand{\le}{\leqslant}
\renewcommand{\epsilon}{\varepsilon}

\newcommand{\C}{\mathbb{C}}
\newcommand{\Z}{\mathbb{Z}}

\usepackage[margin=2cm]{geometry}   % set margin = 2cm
\usepackage{url}
\usepackage{latexsym}
\usepackage{amsmath}
\usepackage{amsfonts}
\usepackage{amssymb}
\usepackage{amsthm}
\usepackage{mathrsfs}
\usepackage{enumerate}

\newcommand{\vect}[1]{\mathbf{#1}}

\renewcommand{\[}{\begin{equation}}
\renewcommand{\]}{\end{equation}}

\newcommand{\Int}{\mathrm{Int}}

\newcommand{\restate}[2]{\medskip
\noindent{\bf #1 (restated).}{\sl #2}}

\definecolor{Blue}{rgb}{0,0,1}

\def\a{{\vect{a}}}
\def\b{{\vect{b}}}

\def\y{{\vect{y}}}
\def\z{{\vect{z}}}
\def\w{{\vect{w}}}
\def\Y{{\vect{Y}}}
\def\Z{{\vect{Z}}}

\begin{document}

\title{A characterization of functions with vanishing averages over products of disjoint sets}

%\tableofcontents

\author{Hamed Hatami \thanks{McGill University. \texttt{hatami@cs.mcgill.ca}. Supported by an NSERC grant.} \and Pooya Hatami \thanks{University of Chicago. \texttt{pooya@cs.uchicago.edu}.} \and Yaqiao Li \thanks{McGill University. \texttt{yaqiao.li@mail.mcgill.ca}.}}

\maketitle

\begin{abstract}
Given  $\alpha_1,\ldots,\alpha_m \in (0,1)$, we characterize all integrable functions $f:[0,1]^m \to \C$ satisfying $\int_{A_1 \times \cdots \times A_m} f =0$ for any collection of disjoint sets $A_1,\ldots,A_m \subseteq [0,1]$ of respective measures $\alpha_1,\ldots,\alpha_m$. We use this characterization to settle some of the conjectures in [S. Janson and V. S\'os, More on quasi-random graphs, subgraph counts and graph limits, arXiv:1405.6808].
\end{abstract}

\newline
{{\sc Keywords:} generalized Walsh expansion, graph limits, quasi-randomness.

\section{Introduction}
Answering a question of Janson and S\'os~\cite[Problem 4.5]{JansonSos}, given  $\alpha_1,\ldots,\alpha_m \in (0,1)$, we characterize all integrable functions $f:[0,1]^m \to \C$ that satisfy
\begin{equation}
\label{eq:mainCondition}
\int_{A_1 \times \cdots \times A_m} f =0
\end{equation}
for every collection of disjoint sets $A_1,\ldots,A_m \subseteq [0,1]$ where $A_1,\ldots,A_m$ are of respective measures $\alpha_1,\ldots,\alpha_m$.  While the question is very natural on its own, it also arises naturally in the study of certain quasi-random properties of graphs.  Indeed this was the original motivation of Janson and S\'os~\cite{JansonSos} for asking and studying this question.

Given a number $p \in (0,1)$, roughly speaking, a graph sequence $\{G_n\}_{n=1}^\infty$  is called $p$-quasi-random if, in the limit, it behaves similar to the sequence of Erd\"os-R\'enyi random graphs $G(|V(G_n)|,p)$. In the seminal works Thomason~\cite{MR930498,MR905280} and Chung, Graham and Wilson~\cite{MR1054011} suggested a rigorous definition of a quasi-random graph sequence, and made a curious observation that  many seemingly different definitions are equivalent, and thus lead to the same notion of quasi-randomness. One that turns out to be particularly useful is the following:

\begin{definition}
\label{def:qrandom}
A graph sequence  $\{G_n\}_{n=1}^\infty$   is $p$-quasi-random if and only if $|V(G_n)| \to \infty$ and $N(F,G_n)=(p^{|E(F)|}+o(1)) |V(G_n)|^{|V(F)|}$ for every graph $F$, where $N(F,G_n)$ denotes the number of \emph{labeled} copies of $F$ in $G_n$ as a subgraph (not necessarily induced).
\end{definition}

A graph sequence  $\{G_n\}_{n=1}^\infty$  is called \emph{convergent}~\cite{MR2274085} if  the normalized subgraph counts $N(F,G_n)/|V_n|^{|V(F)|}$ converge  for every graph $F$.  The limit of  a convergent graph sequence can be represented by a so called \emph{graphon}, which is a symmetric measurable function $W:[0,1]^2 \to [0,1]$. More precisely, given a convergent graph sequence $\{G_n\}_{n=1}^\infty$, there always exists a graphon $W:[0,1]^2 \to [0,1]$ such that for every integer $m>0$ and every graph $F$ with vertex set $\{1,\ldots,m\}$, we have
$$\lim_{n \to \infty} N(F,G_n)/| V(G_n)|^{|V(F)|} = \int_{[0,1]^m} \prod_{(i,j) \in E(F)} W(x_i,x_j) dx_1 dx_2 \cdots dx_m.$$
We denote the integral in the right-hand side by $t(F,W)$. Conversely, for every graphon $W$, one can construct a  graph sequence that converges to $W$ in the above sense. Note that every $p$-quasi-random graph  sequence converges to the constant graphon $W =  p$, where here and in the sequel when we say two functions are equal, we mean they are equal almost everywhere. Hence, often with a bit of work, one can translate various characterizations of $p$-quasi-random graph sequences to statements asserting that the constant graphon $p$ is the unique graphon that satisfies a certain condition. For example, Chung, Graham and Wilson~\cite{MR1054011} showed that  it suffices to require the condition of \cref{def:qrandom} only for two graphs $F=K_2$ and $F=C_4$. In the language of graph limits this corresponds to the fact that the graphon $W=p$ is the unique  graphon that satisfies $t(K_2,W)=p$ and $t(C_4,W)=p^4$.

It is not difficult to see that there is no single graph $F$ such that $t(W,F)=p^{|E(F)|}$ would imply $W=p$. As a substitute, Simonovits and S\'os~\cite{MR1645698} considered the hereditary versions of the subgraph counts, and showed that in fact for every fixed graph $F$, the condition $N(F,G_n[U])=p^{|E(F)|}+o\left(|V(G_n)|^{|V(F)|}\right)$ is satisfied for all subsets $U \subseteq V(G_n)$ if and only if the sequence $\{G_n\}_{n=1}^\infty$ is $p$-quasi-random. Here $G_n[U]$ denotes the subgraph of $G_n$ induced on $U$. In the language of graph limits this is equivalent to saying that given a graph $F$ with vertices $\{1,\ldots,m\}$, the graphon $W=p$ is the only graphon that satisfies
$$\int_{A^m} \prod_{(i,j) \in E(F)} W(x_i,x_j)  dx_1 \cdots dx_m = p^{|E(F)|} \lambda(A)^{m}$$
for all measurable $A \subseteq [0,1]$. Yuster~\cite{MR2676838} showed that given any $\alpha \in (0,1)$, it suffices to require this condition only for $A$ of measure $\alpha$. Shapira~\cite{MR2488748}, Yuster and Shapira~\cite{MR2591049}, and Janson and S\'os~\cite{JansonSos}  considered  the condition
\begin{equation}
\label{eq:generalP}
\int_{A_1 \times \cdots \times A_m} \prod_{(i,j) \in E(F)} W(x_i,x_j)  dx_1 \cdots dx_m = p^{|E(F)|} \prod_{i=1}^m  \alpha_i
\end{equation}
for all disjoint sets $A_1,\ldots,A_m \subseteq [m]$ of respective measures $\alpha_1,\ldots,\alpha_m$. They studied the question that for which   graphs $F$ with vertex set $\{1,\ldots,m\}$ and  sequences $\alpha_1,\ldots,\alpha_m \in (0,1)$ with $\sum_{i=1}^m \alpha_i \le 1$, the graphon $W = p$ is the unique graphon that satisfies \cref{eq:generalP}
for all disjoint sets $A_1,\ldots,A_m \subseteq [0,1]$ of respective measures $\alpha_1,\ldots,\alpha_m$. Following the notation of~\cite{JansonSos}, in this case, we say that $\mathcal{P}(F,\alpha_1,\ldots,\alpha_m)$ is a \emph{quasi-random property}.  Note that this  is equivalent to  \cref{eq:mainCondition} with
$$f=  \left(\prod_{(i,j) \in E(F)} W(x_i,x_j)\right)  - p^{|E(F)|},$$
and thus naturally raises the question~\cite[Problem 4.5]{JansonSos}  that which integrable functions $f:[0,1]^m \to \C$ satisfy \cref{eq:mainCondition} for all disjoint sets $A_1,\ldots,A_m \subseteq [0,1]$ of respective measures $\alpha_1,\ldots,\alpha_m$. We solve this problem in \cref{thm:mainthm}.

As an application of our  \cref{thm:mainthm}, in \cref{cor:quasiprop} and~\cref{thm:twins},  we recover  \cite[Theorem 2.11]{JansonSos}, and furthermore show that when $F$ contains twin vertices, $\mathcal{P}(F,\alpha_1,\ldots,\alpha_m)$ is a quasi-random property. The latter in particular answers \cite[Problem 2.19]{JansonSos} in the affirmative.

Finally as another application of our proof technique, in \cref{thm:symmetric} we solve \cite[Conjecture 9.4]{JansonSos} regarding symmetric functions.

%%%%%%%%%%%%%%%%%%%%%%%%%%%%%%%%%%%%%%%%%%%%%%%%%%%%%%%%%%%%%%%%%%%%%%
\section{Notations and Preliminary Results} \label{sec:notation}
For every natural number $m$, denote $[m]  := \{1, \ldots, m\}$. Let $\lambda$ denote the Lebesgue measure on reals. For $x \in [0,1]^m$ and $S \subseteq [m]$, let $x_S \in [0,1]^S$ denote the restriction of $x$ to the coordinates in $S$.  For disjoint sets $S,T \subseteq [m]$, and $y \in [0,1]^S$ and $z \in [0,1]^T$, let $(y, z)$ denote the unique element  in $[0,1]^{S  \cup T}$ satisfying $(y,z)_S = y$ and $(y,z)_T =z$. For a vector $x \in [0,1]^m$ and an index $i \in [m]$, $x(i) \in [0,1]$ denotes the $i$-th entry of $x$.

For $S \subseteq [m]$, we denote by $\overline{S}:=[m] \setminus S$ the complement of $S$.
Given a function $f:[0,1]^m \to \C$ and $y \in [0,1]^S$, we define $f_y:[0,1]^{\overline{S}} \to \C$ by $f_y: z \mapsto  f(y,z)$ for every $z \in [0,1]^{\overline{S}}$. In the sequel, by an abuse of notation, we sometimes identify a function $f:[0,1]^S \to \C$ with its extension  to $[0,1]^m$  defined as $x \mapsto f(x_S)$ for $x \in [0,1]^m$.

We say that a function $f:[0,1]^m \to \C$ is an \emph{alternating function} with respect to the coordinates in $S \subseteq [m]$ if the   interchange of the values of any two coordinates in $S$ changes the sign of $f$.

Given $\alpha=(\alpha_1,\ldots,\alpha_m) \in [0,1]^m$ with $\sum_{i=1}^m \alpha_i=1$,  call a partition $A_1,\ldots,A_m$ of $[0,1]$ an \emph{$\alpha$-partition} if  $\lambda(A_i)=\alpha_i$ for $i=1,\ldots,m$ and  the boundary of each $A_i$ is of measure $0$. Given subsets $A_1,\ldots,A_m \subseteq [0,1]$ and $S  \subseteq [m]$, let $A_S$ denote  the product $\prod_{i \in S} A_i$.

For a positive integer $m$, let $S_m$ denote the symmetric group of order $m$.

%%%%%%%%%%%%%%%%%%%%%%%%%%%%%%%%%%%%%%%%%%%%%%%%%%%%%%%%%%%%%%%%%%%%
\subsection{Generalized Walsh Expansion}
Our proofs of \cref{thm:mainthm} and \cref{thm:symmetric} use the so called generalized Walsh expansion, which was first defined by Hoeffding in~\cite{MR0026294} (See also~\cite{MR615434}).

\begin{definition}\label{dfn:generalizedwalsh}
The generalized Walsh expansion of  an integrable function $f:[0,1]^m \to \C$ is the expansion $f=\sum_{S\subseteq [m]} F_S$ that satisfies the following two properties.
\begin{itemize}
\item[(i)] For every $S\subseteq [m]$, the function $F_S$ depends only on the coordinates in $S$, i.e. $F_S(x)= F_S(x_S)$;
\item[(ii)] $\int_{[0,1]} F_S(x) dx_i =0,$ for every $S\subseteq [m]$ and every $i\in S$.
\end{itemize}
\end{definition}

We call a function $F_S$ satisfying \cref{dfn:generalizedwalsh}~(i) and (ii) a \emph{generalized Walsh function}. It is not difficult to see that the generalized Walsh expansion is unique and can be computed using the following formula
$$F_S(y) = \sum_{T \subseteq S} (-1)^{|S \setminus T|} \int_{[0,1]^{\overline{T}}} f(y_T,x_{\overline{T}}) dx_{\overline{T}}.$$
In the sequel, for the sake of brevity, we shall often drop the word ``generalized'' from the terms ``generalized Walsh expansion'' and ``generalized Walsh function''.

Given an integrable function $f:[0,1]^m \to \C$, for $0 \le k \le m$, we denote by $f^{\le k} := \sum_{S\subseteq [m], |S| \le k} F_S$ the projection of $f$ to the first $k$ ``levels''. The projections $f^{=k}$,  $f^{\ge k}$, $f^{< k}$, and $f^{> k}$ are defined similarly.
%%%%%%%%%%%%%%%%%%%%%%%%%%%%%%%%%%%%%%%%%%%%%%%%%%%%%%%%%%%%%%%%%%%%

\section{Main Results}

We are now ready to state our results formally. We start by stating our main theorem that characterizes all integrable functions $f:[0,1]^m \to \C$ satisfying  \cref{eq:mainCondition} in the case where $\sum_{i=1}^m \alpha_i=1$. We handle the case $\sum_{i=1}^m \alpha_i<1$ as a consequence of this in \cref{cor:symCase}.

\begin{theorem}[Main Theorem] \label{thm:mainthm}
Let $\alpha_1,\ldots,\alpha_m \in (0,1)$ satisfy $\sum_{i=1}^m \alpha_i =1$. An integrable function  $f:[0,1]^m \to \C$ with the  Walsh expansion $f = \sum_{S \subseteq [m]} F_S$ satisfies
\begin{equation}
\label{eq:mainConditionRepeat}
\int_{A_1 \times \cdots \times A_m} f =0
\end{equation}
for all partitions of $[0,1]$ into disjoint sets $A_1,\ldots,A_m$ of respective measures $\alpha_1,\ldots,\alpha_m$ if and only if
\begin{itemize}
\item[(i)] $F_{\emptyset} = 0$;
\item[(ii)] $F_S$ is  an alternating function (with respect to the coordinates in $S$) for all $S \subseteq [m]$ with $|S| \ge 2$;
\item[(iii)]  For $S\subseteq [m]$, with $1\leq |S|\leq m-1$, and $\ell\in [m]\backslash S$, we have
\begin{equation}
\label{eqn:F_S-formula}
 \frac{1}{\prod_{i \in S} \alpha_i} F_S(x) = \sum_{i \in S} \frac{1}{\prod_{j \in S_i} \alpha_j} F_{S_i} (x^{(i)}),
\end{equation}
where $x^{(i)}$ is obtained from $x=(x_1,\ldots,x_m)$ by swapping  $x_\ell$ and $x_i$, and $S_i := S\cup \{\ell\} \setminus \{i\}$.
\end{itemize}

\end{theorem}

We prove  \cref{thm:mainthm} in \cref{sec:proofThmMain}.
\begin{remark}
\label{rem:one}
Note that \cref{thm:mainthm}~(iii), applied to sets $S$ of size $1$, implies that there exists an integrable function $g:[0,1] \to \C$ with $\int_0^1 g(x)dx=0$ such that $F_{\{i\}}(x) =  \alpha_i g(x)$  for every $i\in [m]$.

\cref{thm:mainthm} provides a way to construct all functions $f$ that satisfy \cref{eq:mainConditionRepeat}.
Indeed one can take any collection of alternating  Walsh functions $F_S: [0,1]^m \to \C$ for $S \ni m$,  and then use \cref{thm:mainthm}~(iii) with $\ell=m$ to define $F_S$ for all other subsets $\emptyset \neq S \subseteq [m]$ accordingly (i.e. all $S \subseteq [m-1]$ of size at least one). Note that the resulting $F_S$ will automatically be Walsh functions and satisfy \cref{thm:mainthm}~(i-iii), and thus the function $f = \sum_{S \subseteq [m]} F_S$ will satisfy \cref{eq:mainConditionRepeat}.
\end{remark}

Next we show how the case $\sum \alpha_i <1$ follows from \cref{thm:mainthm}.

\begin{corollary}[{{\cite[Lemma 4.6]{JansonSos}}}]
\label{cor:symCase}
Let $\alpha_1,\ldots,\alpha_m \in (0,1)$, with $\sum_{i=1}^m \alpha_i <1$. An integrable function  $f:[0,1]^m \to \C$ satisfies
\[
\int_{A_1 \times \cdots \times A_m} f  = 0
\]
for all disjoint $A_1, \ldots, A_m \subseteq [0,1]$ of respective measures $\alpha_1,\ldots,\alpha_m$ if and only if $f=0$ almost everywhere.
The same assertion holds if $f$ is symmetric and $\sum_{i=1}^m \alpha_i =1$, but $(\alpha_1,\ldots,\alpha_m) \neq (1/m,\ldots,1/m)$.
\end{corollary}
\begin{proof}
First consider the  case $\sum_{i=1}^m \alpha_i <1$.
Define $\alpha_{m+1} = 1 - \sum_{i=1}^m \alpha_i$ and apply \cref{thm:mainthm} to the sequence $\alpha_1,\ldots,\alpha_{m+1}$ and the function $\tilde{f}:[0,1]^{m+1} \to \C$ defined as $\tilde{f}:(x_1,\ldots,x_{m+1}) \mapsto f(x_1,\ldots,x_m)$. The assertion now follows from  \cref{thm:mainthm}  as in the Walsh expansion $\tilde{f}=\sum_{S \subseteq [m+1]} \tilde{F}_S$, we have $\tilde{F}_S=0$ for every $S \subseteq [m+1]$ with $m+1 \in S$.

To prove the case where $f$ is symmetric but $(\alpha_1,\ldots,\alpha_m) \neq (1/m,\ldots,1/m)$ note that \cref{thm:mainthm}~(ii) and the symmetry of $f$ imply that $F_S = 0$ for every $S$ with $|S|>1$. Finally, \cref{rem:one} and the symmetry shows $F_S = 0$ for every $S$ of size $1$.
\end{proof}

Following the notation of~\cite{JansonSos}, we say that $\widetilde{\mathcal{P}}(F,\alpha_1,\ldots,\alpha_m)$ is a \emph{quasi-random property} if $W=p$ is the unique solution to
\begin{equation}
\label{eq:SymmCondition}
\frac{1}{m!} \sum_{\sigma \in S_m} \int_{A_{\sigma_1} \times \cdots \times A_{\sigma_m}} \prod_{(i,j) \in E(F)} W(x_i,x_j)  dx_1 \cdots dx_m = p^{|E(F)|} \prod_{i=1}^m  \alpha_i.
\end{equation}
As it is noticed in \cite{JansonSos}, \cref{cor:symCase} has the following consequence.
\begin{corollary}[{{\cite[Theorem 2.11]{JansonSos}}}]
\label{cor:quasiprop}
Let $F$ be a graph with vertex set $\{1,\ldots,m\}$ that contains at least one edge, and let $0<p \le 1$. Furthermore, let $(\alpha_1,\ldots,\alpha_m)$ be a vector of positive numbers with $\sum_{i=1}^m \alpha_i \le 1$.
\begin{itemize}
\item[(i)] If $(\alpha_1,\ldots,\alpha_m) \neq (1/m,\ldots,1/m)$, then $\widetilde{\mathcal{P}}(F,\alpha_1,\ldots,\alpha_m)$ is a quasi-random property.
\item[(ii)] If $\sum_{i=1}^m \alpha_i < 1$ then $\mathcal{P}(F,\alpha_1,\ldots,\alpha_m)$ is a quasi-random property.
\end{itemize}
\end{corollary}

We call two vertices in a graph \emph{twins} if they share the same neighbors (and thus
there is no edge between them). Next we use \cref{thm:mainthm} to prove a theorem about graphs containing twin vertices. This in particular solves \cite[Problem 2.19]{JansonSos}
regarding quasi-random properties of stars by noting that stars with at least three
vertices always contain twins.

\begin{theorem} \label{thm:twins}
Let $F$ be a graph containing twins, and let $0< p \le 1$, then $\mathcal{P}(F, \alpha_1,\ldots,\alpha_m)$ is a quasi-random property for all $\alpha_1,\ldots,\alpha_m \in (0,1)$ with $\sum_{i=1}^m \alpha_i \le 1$.
\end{theorem}

\begin{proof}
The case $\sum \alpha_i <1$ follows from \cref{cor:quasiprop}. It remains to establish the case $\sum \alpha_i = 1$. Let $f := \prod_{(i,j)\in E(F)} W(x_i, x_j) - p^{|E(F)|}$. We will show that if $\int_{A_{[m]}} f = 0$ for all $\alpha$-partitions, then  $W = p$ almost everywhere.

Without loss of generality, assume $v_{m-1}, v_m$ are twins in $F$, and $v_1, \ldots, v_r$, $r \le m-2$, are their common neighbors. Let  $f= \sum_{S \subseteq [m]} F_S$ be the Walsh expansion of $f$. Therefore, $f$ can be written in the following form
\begin{align} \label{eqn:twins}
f
&= \left(\prod_{\substack{(i,j)\in E(F)\\ i,j\in [m-2]}} W(x_i,x_j) \right)\Big( \prod_{i=1}^r W(x_i, x_{m-1}) W(x_i, x_m) \Big) - p^{|E(F)|} \nonumber \\
&= \sum_{S \subseteq [m-2]} (F_S + F_{S\cup \{m-1\}} + F_{S\cup \{m\}}+F_{S\cup \{m-1,m\}}).
\end{align}
We claim that for every $S\subseteq [m-2]$, $F_{S\cup \{m-1,m\}} =0$ almost everywhere. Indeed since $v_{m-1}, v_m$ are twins,  $F_{S\cup \{m-1,m\}}$ is symmetric with respect to the two coordinates $x_{m-1}$ and $x_m$, and on the other hand by \cref{thm:mainthm}~(ii), $F_{S\cup \{m-1,m\}}$ is also an  alternating function with respect to those coordinates. Hence, $F_{S\cup \{m-1,m\}} = 0$ almost everywhere.

Fixing $x_1,\ldots,x_{m-2}$ and integrating \cref{eqn:twins} with respect to $x_{m-1},x_m \in [0,1]$, we obtain that for almost every $x_1,\ldots,x_{m-2}$,
\begin{equation}\label{eq:LHS}
p^{|E(F)|} + \sum_{S\subseteq [m-2]} F_S = \prod_{\substack{(i,j)\in E(F)\\ i,j\in [m-2]}} W(x_i,x_j) \left( \int_{[0,1]}  \prod_{i=1}^r W(x_i,a) da \right)^2.
\end{equation}
Next we would like to replace the same value $a$ for both $x_{m-1}$ and $x_m$ in \cref{eqn:twins}, and then integrate with respect to $a$. However, since $F_{S\cup \{m-1,m\}}=0$ only almost everywhere, we need to consider the limit instead. More precisely we deduce the following from \cref{eqn:twins} and the fact that $F_{S\cup \{m-1,m\}}=0$ almost everywhere:  For almost all $x_1,\ldots,x_{m-2},a \in [0,1]$, denoting $y:=(x_{[m-2]},a,a) \in [0,1]^m$, we have
\begin{align*}
\lim_{\epsilon_1, \epsilon_2 \rightarrow 0^+}\frac{1}{\epsilon_1 \epsilon_2} \int_{B_{\epsilon_1/2}(a)}\int_{B_{\epsilon_2/2}(a)} f dx_{m-1} dx_m &= \prod_{\substack{(i,j)\in E(F)\\ i,j\in [m-2]}} W(x_i,x_j)  \prod_{i=1}^r W(x_i,a)^2 - p^{|E(F)|}
\\ &= \sum_{S \subseteq [m-2]} F_S(y) + F_{S\cup \{m-1\}}(y)+ F_{S\cup \{m\}}(y).
\end{align*}
Integrating this with respect to $a$, we obtain that for almost all $x_{[m-2]}$,
\begin{equation}\label{eq:RHS}
p^{|E(F)|} + \sum_{S\subseteq [m-2]} F_S = \prod_{\substack{(i,j)\in E(F)\\ i,j\in [m-2]}} W(x_i,x_j) \int_{[0,1]}  \prod_{i=1}^r W(x_i,a)^2 da.
\end{equation}
Hence $(\ref{eq:LHS}) = (\ref{eq:RHS})$ for almost all $x_{[m-2]}$, and then the equality condition of Cauchy-Schwarz implies that for almost all $x_{[m-2]}$,  $\prod_{\substack{(i,j)\in E(F)\\ i,j\in [m-2]}} W(x_i,x_j)   \prod_{i=1}^r W(x_i,a)$ does not depend on $a$.  It follows that $$\prod_{\substack{(i,j)\in E(F)\\ i,j\in [m-2]}} W(x_i,x_j)  \prod_{i=1}^r W(x_i,x_{m-1}) W(x_i,a)$$ does not depend on $a$ for almost all
$x_{[m-1]}$. Hence for every $\alpha$-partition $A_{[m]}$ and every $B \subseteq A_m$ with $\lambda(B)>0$, we have
\[
\int_{A_{[m-1] \times B}} \prod_{(i,j) \in E(F)}  W(x_i,x_j) dx_{[m]}
= \frac{\lambda(B)}{\alpha_m} \int_{A_{[m]}} \prod_{(i,j) \in E(F)}  W(x_i,x_j) dx_{[m]}
= p^{|E(F)|} \lambda(B) \prod_{i=1}^{m-1} \alpha_i. \nonumber
\]
Now \cref{cor:quasiprop}~(ii) implies that $W=p$ almost everywhere.
\end{proof}

Finally, we state our theorem about symmetric functions which in particular solves~\cite[Conjecture 9.4]{JansonSos}.

\begin{theorem}\label{thm:symmetric}
Let $\alpha\in (0,1)$, and $0 \le r \le m$ be an integer. A symmetric integrable function $f:[0,1]^m \rightarrow \C$ satisfies $\int_{A^{m-r} \times (\overline{A})^{r}} f =0$ for every $A\subset [0,1]$ with $\lambda(A)=\alpha$ if and only if at least one of the following two cases holds.
\begin{itemize}
\item[(i)] $f=0$ almost everywhere.
\item[(ii)] For $K := K(m,r,\alpha) = \{k \in [m]: \sum_{i=0}^k {m-r \choose k-i}
{r \choose i}   \left(\frac{-\alpha}{1-\alpha}\right)^{i}= 0\}$, we have
$$
f(x_1,\ldots,x_m)= \sum_{k \in K} \sum_{S\subseteq [m], |S|=k} g_k(x_S)
$$
where $g_k:[0,1]^k \to \C$ are symmetric functions satisfying $\int g_k(x_1,\ldots,x_k) dx_i = 0$ for every $i \in [k]$.
\end{itemize}
\end{theorem}

Note that \cref{thm:symmetric}~(ii) means that in the  Walsh expansion $f= \sum_{S \subseteq [m]} F_S$, we have $F_S=0$ if $|S| \not\in K(m,r,\alpha)$ and $F_S(x)=g_k(x_S)$ if $|S|=k \in  K(m,r,\alpha)$. We shall not  venture to characterize the sets $K(m,r,\alpha)$. However we remark that these sets can contain more than one element, as for example, it is not difficult to see that  $K(6,3,\frac{1}{2})=\{1,3,5\}$.   Thus, in general, the  Walsh expansion of $f$ can be supported on more than one ``level''.

The case $m=3$ and $r=1$ of \cref{thm:symmetric} was conjectured in~\cite[Conjecture 9.4]{JansonSos}. Note that if  $r=1$ in \cref{thm:symmetric}, then  we have $K=\{k\}$ if  $\alpha=\frac{m-k}{m}$  and $K(m,r,\alpha) = \emptyset$ if $\alpha$ is not of the form $\frac{m-k}{m}$. We state this case separately as \cref{cor:symmetric_r_one}

\begin{corollary}[{{\cite[Conjecture 9.4]{JansonSos}}}]\label{cor:symmetric_r_one}
Let $\alpha\in (0,1)$. A   symmetric integrable function $f:[0,1]^m \rightarrow \C$ satisfies $\int_{A^{m-1} \times \overline{A}} f =0$ for every $A\subset [0,1]$ with $\lambda(A)=\alpha$ if and only if at least one of the following two cases holds.
\begin{itemize}
\item[(i)] $f=0$ almost everywhere.
\item[(ii)] $\alpha= \frac{m-k}{m}$ for some $k\in [m-1]$ and
$$
f(x_1,\ldots,x_m)= \sum_{S\subseteq [m], |S|=k} g(x_S)
$$
where $g:[0,1]^k \to \C$ is a symmetric function satisfying $\int g(x_1,\ldots,x_k) dx_i = 0$ for every $i \in [k]$.
\end{itemize}
\end{corollary}

\subsection{Proof Technique: A first variation argument} \label{section:derivative}

In this short section we prove the main step  used in the proofs of \cref{thm:mainthm} and \cref{thm:symmetric}.
Let us recall the following form of the Lebesgue differentiation theorem.
\begin{lemma}\label{lebesgue-differentiation-thm}
Let $g:[0,1]\rightarrow \C$ be Lebesgue integrable and let $x\in [0,1]$ be a Lebesgue point of $g$. Then
$$
g(x)= \lim_{\epsilon\rightarrow 0^+} \frac{1}{\epsilon}\int_{B_{\epsilon/2}(x)} g(y)dy,
$$
where $B_{\epsilon/2}(x)$ is the ball of radius $\epsilon/2$ around $x$.
\end{lemma}

Let $\alpha_1, \ldots ,\alpha_m\in (0,1)$  and suppose that $A_1, \ldots, A_m \subseteq [0,1]$ are of measures $\alpha_1,\ldots,\alpha_m$, respectively. Consider $K \subseteq [m]$, and given $\y, \z \in [0,1]^K$ and $t \ge 0$, set
\begin{equation}
\label{eq:At}
A_i(t) := A_i\cup B_{t/2}(\z(i)) \setminus B_{t/2}(\y(i))
\end{equation}
for every $i \in K$, and $A_i(t):=A_i$ for $i \in \overline{K}$. Now consider an integrable function $f:[0,1]^m \to \C$, and define
$$F(t) = \int_{A_1(t) \times \cdots \times A_m(t) } f.$$
It follows from \cref{lebesgue-differentiation-thm} that for almost every $\y \in \Int(A_{K})$ and $\z \in \Int(\prod_{i \in K} \overline{A_i})$, we have
\begin{equation}
\label{eq:differentation}
\left. \frac{dF(t)}{dt} \right|_{0^+} = \sum_{i \in K} \int_{A_{[m] \setminus \{i\}}} \left(f_{\z(i)} - f_{\y(i)}  \right)  dx_{[m] \setminus \{i\}} =\sum_{i \in K} \frac{1}{\alpha_i}\int_{A_{[m]}} \left(f_{\z(i)} - f_{\y(i)} \right)  dx_{[m]}.
\end{equation}
Let us introduce the notation
$$\partial^i_{\y,\z} f := \frac{f_{\z(i)} - f_{\y(i)}}{\alpha_i},$$
and
$$\partial^K_{\y,\z}f :=\sum_{i \in K} \partial^i_{\y,\z} f = \sum_{i  \in K} \frac{f_{ \z(i)} - f_{\y(i)}}{\alpha_i},$$
so that
$$\left. \frac{dF(t)}{dt} \right|_{0^+} =  \int_{A_{[m]}} \partial^K_{\y,\z} f.$$
Further, suppose $\Y = (\y_1, \ldots, \y_k)$ and $\Z = (\z_1, \ldots, \z_k)$  where $\y_i,\z_i \in [0,1]^K$ for $i=1,\ldots,k$. Define
$$\partial^K_{\Y,\Z} f := \partial^K_{\y_k,\z_k}  \cdots  \partial^K_{\y_1,\z_1} f.$$
Note that when $g:[0,1]^m \to \C$ does not depend on the $i$-th coordinate, then $\partial^{i}_{\y,\z} g = 0$.  Combining this and the fact that $\partial_{\y,\z}^i f$ does not depend on the $i$-th coordinate, we conclude that for any  Walsh function $F_S$, and  any $\Y,\Z \in ([0,1]^K)^k$, we have
$$\partial_{\vect{Y},\Z}^K F_S  =   \sum_{j_1,\ldots,j_k \in  S \cap K \atop  |\{j_1,\ldots,j_k\}|=k} \partial_{\y_k,\z_k}^{j_k}  \cdots  \partial_{\y_1,\z_1}^{j_1} F_S.$$
Expanding this formula leads to the following lemma which is central to the proofs of both \cref{thm:mainthm} and \cref{thm:symmetric}.
\begin{lemma} \label{Lemma:k-Operators}
Consider $S,K \subseteq [m]$, and let $F_S:[0,1]^m \to \C$ depend only on the coordinates in $S$. Given any $\Y,\Z \in ([0,1]^K)^k$, we have
$$\partial_{\Y,\Z}^K F_S =
\sum_{D \subseteq S \cap K \atop |D|=k} \frac{1}{\prod_{i \in D} \alpha_i} \sum_{\pi:D \xrightarrow{1:1} [k]}\sum_{B \subseteq  [k]} (-1)^{|B|} F_S(\w,\cdot)
$$
where $\w =\w_{B,\pi} \in [0,1]^{D}$ defined as
$$\w(t) = \left\{
\begin{array}{lcl}
\y_{\pi(t)}(t) & \qquad &  \pi(t) \in B, \\
\z_{\pi(t)}(t) & & \pi(t) \in [k] \setminus B.
\end{array} \right.
$$
In particular, $\partial_{\Y,\Z}^K F_S$ is equal to zero if $|S \cap K| < k$, and is a constant if $|S|=|S\cap K|=k$.
\end{lemma}

In the sequel,  $\partial_{\Y,\Z}$ and $\partial_{\y,\z}$ are respectively short forms for $\partial_{\Y,\Z}^{[m]}$ and $\partial_{\y,\z}^{[m]}$.

\section{Proof of \cref{thm:mainthm}}
\label{sec:proofThmMain}
Throughout this section fix $\alpha=(\alpha_1,\ldots,\alpha_m) \in (0,1)^m$. Given a subset $S \subseteq [m]$, denote $\alpha(S):=\prod_{i \in S} \alpha_i$. Let us recall the statement of \cref{thm:mainthm}.

\restate{\Cref{thm:mainthm}}{
Let $\alpha_1,\ldots,\alpha_m \in (0,1)$ satisfy $\sum_{i=1}^m \alpha_i =1$. An integrable function  $f:[0,1]^m \to \C$ with the  Walsh expansion $f = \sum_{S \subseteq [m]} F_S$ satisfies
\begin{equation}
\label{eq:mainConditionRepeat2}
\int_{A_1 \times \cdots \times A_m} f =0
\end{equation}
for all partitions of $[0,1]$ into disjoint sets $A_1,\ldots,A_m$ of respective measures $\alpha_1,\ldots,\alpha_m$ if and only if
\begin{itemize}
\item[(i)] $F_{\emptyset} = 0$;
\item[(ii)] $F_S$ is an alternating function (with respect to the coordinates in $S$) for all $S \subseteq [m]$ with $|S| \ge 2$;
\item[(iii)]  For $S\subseteq [m]$, with $1\leq |S|\leq m-1$, and $\ell\in [m]\backslash S$, we have
\begin{equation}
\label{eqn:F_S-formula}
 \frac{1}{\prod_{i \in S} \alpha_i} F_S(x) = \sum_{i \in S} \frac{1}{\prod_{j \in S_i} \alpha_j} F_{S_i} (x^{(i)}),
\end{equation}
where $x^{(i)}$ is obtained from $x=(x_1,\ldots,x_m)$ by swapping  $x_\ell$ and $x_i$, and $S_i := S\cup \{\ell\} \setminus \{i\}$.
\end{itemize}
}

We divide the proof into  two sections, the ``if'', and the ``only if'' parts. Firstly, we prove the following lemma which will be useful in both directions. Recall that $f^{=k} := \sum_{S \subseteq [m] \atop |S|=k} F_S$.

\begin{lemma} \label{lemmma:integration-on-k-level-equal-zero}
Given any fixed $1 \le k \le m$, assume that  \cref{thm:mainthm}~(ii) and (iii) hold for all $F_S$ such that $|S|=k$. Then for all $\alpha$-partitions $A_1,\ldots,A_m$, we have
\[
\int_{A_{[m]}} f^{=k} = 0.
\]
\end{lemma}
Note that (ii) is void in the case of $k=1$ and thus holds trivially.
\begin{proof}
 Consider an $\alpha$-partition $A_1, \ldots, A_m$. For the given $k$, for any $S\subseteq [m]$ with $m \in S$ and $|S|=k$, because $\sum_{i=1}^m \int_{A_i}  F_S(x) d x_m = \int_{[0,1]} F_S(x) d x_m = 0$, we have
$$\int_{A_{[m]}} F_S = - \sum_{i=1}^{m-1} \int_{A_i} \left( \int_{A_{[m-1]}} F_S dx_{[m-1]}\right) dx_m.$$
\cref{thm:mainthm}~(ii) says $F_S$ is an alternating function, which implies $\int_{A_i} \int_{A_i} F_S dx_i dx_m = 0$ if $i \in S$. Hence
$$ \int_{A_{[m]}} F_S = - \sum_{i \not\in S} \int_{A_i} \int_{A_{[m-1]}} F_S dx_{[m-1]} dx_m.$$
Consequently
\begin{align*}
& \frac{1}{\alpha([m])} \int_{A_{[m]}} \sum_{S \subseteq [m] \atop |S|=k} F_S dx_{[m]}  =  \sum_{S \subseteq [m-1] \atop |S|=k} \frac{1}{\alpha([m])} \int_{A_{[m]}} F_S dx_{[m]} + \sum_{S \subseteq [m], S \ni m \atop |S|=k} \frac{1}{\alpha([m])} \int_{A_{[m]}} F_S dx_{[m]}
\\
&=  \sum_{S \subseteq [m-1] \atop |S|=k} \frac{1}{\alpha(S)} \int_{A_{S} } F_S dx_S - \sum_{S \subseteq [m], S \ni m \atop |S|=k} \frac{1}{\alpha([m])} \sum_{i \not\in S} \int_{A_i} \int_{A_{[m-1]}} F_S dx_{[m-1]} dx_m
\\
&=  \sum_{S \subseteq [m-1] \atop |S|=k} \frac{1}{\alpha(S)} \int_{A_{S} } F_S dx_S - \sum_{S \subseteq [m], S \ni m \atop |S|=k} \frac{1}{\alpha(S)} \sum_{i \not\in S} \int_{A_i} \int_{A_{S\backslash \{m\}}} F_S dx_{S\backslash \{m\}} dx_m
\\
&=  \sum_{S \subseteq [m-1] \atop |S|=k} \frac{1}{\alpha(S)} \int_{A_{S} } F_S (x) dx_S - \sum_{S \subseteq [m], S \ni m \atop |S|=k} \frac{1}{\alpha(S)} \sum_{i \not\in S} \int_{A_{S\backslash \{m\}} \times A_i} F_S (x) dx_{S\backslash \{m\}} dx_m
\\
&= \sum_{S \subseteq [m-1] \atop |S|=k} \left( \frac{1}{\alpha(S)} \int_{A_{S} } F_S (x) dx_S -  \sum_{i \in S} \frac{1}{\alpha(S_i)} \int_{A_S} F_{S_i} (x^{(i)}) dx_S \right)=0,
\end{align*}
where the last equality uses \cref{thm:mainthm}~(iii), with $S\subseteq [m-1]$ and $\ell=m$.
\end{proof}

\subsection{Proof of \cref{thm:mainthm}, the ``if'' part}
Using \cref{lemmma:integration-on-k-level-equal-zero}, $\int_{A_{[m]}} f^{=k} = 0$ for all $1 \le k \le m$ by (ii) and (iii). Hence (i-iii) imply $\int_{A_{[m]}} f = 0$.

\subsection{Proof of \cref{thm:mainthm}, the ``only if'' part}
For $\vect{y}=(y_1,\ldots,y_m) \in [0,1]^m$ and $\sigma \in S_m$, let $\y_\sigma=(y_{\sigma(1)},\ldots,y_{\sigma(m)}) \in [0,1]^m$ be obtained by rearranging the coordinates of $\y$ according to $\sigma$. Given $\sigma \in S_m$, define $K_\sigma := \{i\in[m] : i \neq \sigma(i)\}$.  Let us denote $\partial_{\y,\sigma}^j := \partial_{\y, \y_\sigma}^j$,   $\partial_{\y,\sigma} := \partial^{K_\sigma}_{\y, \y_\sigma}$ and $\partial_{\Y,\sigma} := \partial^{K_\sigma}_{\Y,\Y_\sigma}$, where  for $\Y=(\y_1,\ldots,\y_k) \in ([0,1]^m)^k$, we have $\Y_\sigma= ((\y_1)_\sigma,\ldots,(\y_k)_\sigma)$.

To prove the theorem we will  use induction on $|S|$ to show that $F_S$ satisfies  \cref{thm:mainthm}~(ii) and (iii)  for all $S$ with $|S| \ge 1$. \cref{thm:mainthm}~(i) then follows from \cref{thm:mainthm}~(ii-iii) and \cref{eqn:F_S-formula}. Let $k \ge 1$, and assume (ii) and (iii) hold for all $F_S$ such that $k+1 \le |S| \le m$. By \cref{lemmma:integration-on-k-level-equal-zero} we have
\begin{equation}
\label{eq:project-down}
\int_{A_{[m]}} f^{\le k}=\int_{A_{[m]}} f = 0,
\end{equation}
for every $\alpha$-partition $A_1,\ldots,A_m$.

Consider an $\alpha$-partition $A_1,\ldots,A_m$, and let  $\Y = (\y_1, \ldots, \y_k)$ where $\y_i = (y_{i1}, \ldots, y_{im}) \in \Int(A_{[m]})$ are generic points for $i=1,\ldots,k$.

Note that for every sufficiently small $t>0$, the sets $A_1(t),\ldots, A_m(t)$ defined as in \cref{eq:At}, with $\y=\y_1$ and $\z=(\y_1)_\sigma$, and any $\sigma \in S_m$ form an $\alpha$-partition of $[0,1]$.   Hence  by \cref{eq:project-down}, we have $F(t) = \int_{A_1(t) \times \cdots \times A_m(t) } f^{\le k} =0$ for sufficiently small $t$. Consequently  $\left. \frac{dF(t)}{dt} \right|_{0^+} =0$ which in turn implies that
$\int_{A[m]} \partial_{\y_1,\sigma } f^{\le k}=0$. Replacing $f^{\le k}$ with $\partial_{\y_1,\sigma}  f^{\le k}$ and repeating the above argument we conclude that
\begin{equation}
\label{eq:derivatives-zero}
\int_{A_{[m]}} \partial_{\Y,\sigma }f^{\le k}= \int_{A_{[m]}} \partial_{\y_k,\sigma} \cdots \partial_{\y_1,\sigma }f^{\le k}=0,
\end{equation}
for every $\sigma \in S_m$, every $\alpha$-partition $A_1,\ldots,A_m$, and almost every set of points $\y_1,\ldots,\y_k \in \Int(A_{[m]})$.

\noindent {\bf  \cref{thm:mainthm}~(ii):} Let $S \subseteq [m]$ be of size $k \ge 2$. Without loss of generality assume that $S=[k]$.  Setting $\sigma=(1\ 2\ \cdots \ k) \in S_m$, we have $K_\sigma = [k] = S$. By \cref{Lemma:k-Operators}, \cref{eq:derivatives-zero} simplifies to
$$
0 =\int_{A_{[m]}}\partial_{\Y,\sigma}f^{\le k}  =  \int_{A_{[m]}} \partial_{\Y,\sigma} F_S = \frac{\alpha([m])}{\alpha([k])} \sum_{\pi:[k] \xrightarrow{1:1} [k]}\sum_{B \subseteq  [k]} (-1)^{|B|} F_S(\w),
$$
where  $\w =\w_{B,\pi} \in [0,1]^{[k]}$ is defined as
$$\w(t) = \left\{
\begin{array}{lcl}
y_{\pi(t)t}, & \qquad &  \pi(t) \in B, \\
y_{\pi(t)\sigma(t)}, & & \pi(t) \in [k] \setminus B.
\end{array} \right.
$$
Hence, for almost every $\y_1,\ldots,\y_k \in [0,1]^m$, we have
\begin{equation}
\label{eq:diffF_T}
\sum_{\pi:[k] \xrightarrow{1:1} [k]}\sum_{B \subseteq  [k]} (-1)^{|B|} F_S (\w) = 0.
\end{equation}
Let us fix the $k$ entries $\{y_{12}, y_{22}, y_{33}, \ldots, y_{kk}\}$ among $m\times k$ entries in $\vect{Y}$. We claim that in (\ref{eq:diffF_T}), there are only two terms containing these $k$ entries simultaneously: $F_S (y_{12}, y_{22}, y_{33}, \ldots, y_{kk})$ and $F_S (y_{22}, y_{12}, y_{33}, \ldots, y_{kk})$, corresponding respectively to $(\pi(1),\ldots,\pi(k)):=(1,\ldots,k), B=[k]\backslash \{1\}$ and $(\pi(1),\pi(2),\ldots,\pi(k)):=(2,1,3,4\ldots,k), B=[k]\backslash \{2\}$. In particular, the cardinalities of these two $B$s are the same, hence these two terms are of the same sign.

To see the claim, observe that $2$ appears twice as the column index in the $k$ entries $\{y_{12}, y_{22}, y_{33},  \ldots, y_{kk}\}$, and hence by the definition of $\w(t)$, we must have either $\w(1)=y_{12}, \w(2)=y_{22}$ or $\w(1)=y_{22}, \w(2)=y_{12}$. It is then easy to see that, by our choice of $\sigma=(1\ 2\ \cdots \ k)$, the values for the remaining entries of $\w(t)$ are uniquely determined as $\w(t)=y_{tt}, 3\le t \le k$. The permutation $\pi$ and the set $B$ are then determined accordingly.

Thus by \cref{dfn:generalizedwalsh}~(ii), integrating \cref{eq:diffF_T} over all the variables $\{y_{ij}: i \in [m], j \in [k]\} \setminus \{y_{12}, y_{22}, y_{33}, \ldots, y_{kk}\}$ we get
$$F_S (y_{12}, y_{22}, y_{33}, \ldots, y_{kk}) + F_S (y_{22}, y_{12}, y_{33}, \ldots, y_{kk})=0.$$
This shows that $F_S$ is an alternating function with respect to the first two coordinates. The condition with respect to the other coordinates can be shown similarly.

\noindent {\bf \cref{thm:mainthm}~(iii):}
It remains to show that $F_S$ satisfies \cref{thm:mainthm}~(iii).  By symmetry it suffices to prove the statement for $\ell=m$. Again without loss of generality assume $S=[k]$, and now let $\rho=(1,2,\ldots,k,m) \in S_m$. Since $K_\rho=\{1,\ldots,k\} \cup \{m\}$, denoting $S_0:=S$ and defining $S_1,\ldots,S_k$ as in \cref{thm:mainthm}~(iii),  \cref{eq:derivatives-zero} reduces to
\begin{eqnarray*}
0 &=&\int_{A_{[m]}}\partial_{\Y,\rho}f^{\le k}  =  \int_{A_{[m]}} \partial_{\Y,\rho} \sum_{i=0}^kF_{S_i} \nonumber \\
&=& \alpha([m]) \sum_{i=0}^k \sum_{\pi:S_i \xrightarrow{1:1} [k]}\sum_{B \subseteq  [k]} \frac{1}{\alpha(S_i)} (-1)^{|B|} F_{S_i}(\w),
\end{eqnarray*}
where  $\w =\w_{i,B,\pi} \in [0,1]^{S_i}$ is defined as
$$\w(t) = \left\{
\begin{array}{lcl}
y_{\pi(t)t}, & \qquad &  \pi(t) \in B, \\
y_{\pi(t)\rho(t)}, & & \pi(t) \in S_i \setminus B.
\end{array} \right.
$$
Hence for almost every $\y_1,\ldots,\y_k \in [0,1]^m$, we have
\begin{equation}
\label{eq:diffF_T_iv}
\sum_{i=0}^k \sum_{\pi:S_i \xrightarrow{1:1} [k]}\sum_{B \subseteq  [k]}  \frac{1}{\alpha(S_i)} (-1)^{|B|} F_{S_i}(\w) = 0,
\end{equation}
where  $\w =\w_{i,B,\pi} \in [0,1]^{S_i}$ is as above.

This time we fix the $k$ variables $ y_{1 1}, y_{2 2}, \ldots, y_{k k} $ among the $k\times m$  entries of $\vect{Y}$. In \cref{eq:diffF_T_iv}, using the definition of $\rho$ and a similar argument as for the previous claim, those terms containing exactly these $k$ points as their coordinates are as follows:
\begin{itemize}
\item The term: $(-1)^k F_{S_0}(y_{11}, y_{22}, \ldots, y_{kk})$, corresponding to $(\pi(1),\ldots,\pi(k))=(1,\ldots,k)$ and $B=[k]$.
\item For each $1 \le i \le k$, there is  one such term: $(-1)^{k-i} F_{S_i}(\w)$ with
$$\w(j)=
\left\{
\begin{array}{lcl}
y_{j+1,j+1}, & & 1 \le j \le i-1, \\
y_{jj}, &\qquad  & i+1 \le j \le k, \\
y_{11}, & & j=m,
\end{array}
\right.
$$
for $j \in S_i$, corresponding to $B=\{i+1,\ldots,k\}$, and $\pi$ defined as $\pi(m)=1$, $\pi(j)=j+1$ for $1\le j \le i-1$, and $\pi(j)=j$ for $i+1 \le j \le k$. Since $F_{S_i}$ is an alternating function,
$$(-1)^{k-i} F_{S_i}(\w)=(-1)^{k-i} (-1)^{i-1} F_{S_i}(\w') = (-1)^{k-1} F_{S_i}(\w'),$$
where for $j \in S_i$, $\w'$ is defined as
$$\w'(j)=
\left\{
\begin{array}{lcl}
y_{jj}, & & j \neq m, \\
y_{ii}, &\qquad  & j =m.
\end{array}
\right.
$$
\end{itemize}

Hence fixing $y_{11},\ldots,y_{kk}$ and integrating with respect to the other $(m-1)k$ entries of $\Y$, by \cref{dfn:generalizedwalsh}~(ii), \cref{eq:diffF_T_iv} reduces to \cref{thm:mainthm}~(iii).

\section{Proof of \Cref{thm:symmetric}}
In this section we will prove \Cref{thm:symmetric}.

\restate{\Cref{thm:symmetric}}{
Let $\alpha\in (0,1)$, and $0 \le r \le m$ be an integer. A symmetric integrable function $f:[0,1]^m \rightarrow \C$ satisfies $\int_{A^{m-r} \times (\overline{A})^{r}} f =0$ for every $A\subset [0,1]$ with $\lambda(A)=\alpha$ if and only if at least one of the following two cases holds.
\begin{itemize}
\item[(i)] $f=0$ almost everywhere.
\item[(ii)] For $K := K(m,r,\alpha) = \{k \in [m]: \sum_{i=0}^k {m-r \choose k-i}
{r \choose i}   \left(\frac{-\alpha}{1-\alpha}\right)^{i}= 0\}$, we have
$$
f(x_1,\ldots,x_m)= \sum_{k \in K} \sum_{S\subseteq [m], |S|=k} g_k(x_S)
$$
where $g_k:[0,1]^k \to \C$ are symmetric functions satisfying $\int g_k(x_1,\ldots,x_k) dx_i = 0$ for every $i \in [k]$.
\end{itemize}
}

Since $f$ is symmetric, the  Walsh expansion $f=\sum_{S \subseteq [m]} F_S$ has the following structure. Every $F_S$ is symmetric with respect to the coordinates in $S$, and furthermore for every $0 \le k \le m$ and every $S \subseteq [m]$ with $|S|=k$, we have
$$F_S(a_1,\ldots,a_k) = F_{[k]}(a_1,\ldots,a_k).$$
Note that
$$
\int_{A^{m-r} \times (\overline{A})^r} F_S = (-1)^{|S\cap \{m-r+1,\ldots,m\}|} \left(\frac{1-\alpha}{\alpha}\right)^{r-|S\cap \{m-r+1,\ldots,m\}|}\int_{A^m} F_S.
$$
We conclude that
$$
\int_{A^{m-r} \times (\overline{A})^r} f =  \left(\frac{1-\alpha}{\alpha}\right)^r \sum_{k=0}^{m} \sum_{i=0}^k {m-r \choose k-i}
{r \choose i}   \left(\frac{-\alpha}{1-\alpha}\right)^{i} \int_{A^m} F_{[k]}.
$$
This verifies the ``if'' part of \Cref{thm:symmetric}.  It remains to prove the ``only if'' part.

Let
$$F=:  \sum_{k=0}^{m}\left(\sum_{i=0}^k {m-r \choose k-i}
{r \choose i}   \left(\frac{-\alpha}{1-\alpha}\right)^{i} \right) F_{[k]}.$$
Under the assumption of the theorem, we have $\int_{A^m} F = 0$ for every  $A \subseteq [0,1]$ with $\lambda(A)=\alpha$. Now similar  to \Cref{section:derivative} and the proof of \cref{thm:mainthm}, we use the fact that the integral needs to remain $0$ under small modifications of $A$ that do not change its measure.

Fix $A\subset [0,1]$ with $\lambda(A)=\alpha$ and nonempty interior and exterior.  In order to use the notation of \Cref{section:derivative}, define $A_1:=\cdots:=A_{m}:=A$.    Consider $\a=(a_1,\ldots,a_k) \in \Int(A)^k$ and $b=(b_1,\ldots,b_k) \in \Int(\overline{A})^k$, and let $\y_i:=(a_i,\ldots,a_i) \in [0,1]^m$ and $\z_i:=(b_i,\ldots,b_i) \in [0,1]^m$.  We conclude that for almost every $a_1,\ldots,a_k \in \Int(A)$ and $b_1,\ldots,b_k \in \Int(\overline{A})$, we have
\begin{equation}
\label{eq:sym_derivZero}
\int_{A_{[m]}} \partial_{\y_k, \z_k}  \cdots  \partial_{\y_1, \z_1} F = 0.
\end{equation}

\begin{claim}
We have
$$
\partial_{\y_k, \z_k}  \cdots  \partial_{\y_1, \z_1}  F_{[k]} = \alpha^{-k} k! \sum_{B \subseteq  [k]}  (-1)^{|B|}  F_{[k]}(\a_B,\b_{[k] \setminus B}).
$$
Furthermore,
$\partial_{\y_k, \z_k}  \cdots  \partial_{\y_1, \z_1} F_{[\ell]} = 0$ if $\ell <k$.
\end{claim}
\begin{proof}
The claim is an easy consequence of \cref{Lemma:k-Operators}. By this lemma,   $\partial_{\y_k, \z_k}  \cdots  \partial_{\y_1, \z_1}  F_T =0$ if $|T|<k$, and moreover
\begin{equation}
\label{eq:partialSymWalsh}
\partial_{\y_k, \z_k}  \cdots  \partial_{\y_1, \z_1} F_{[k]}= \frac{1}{\alpha^k} \sum_{\pi:[k] \xrightarrow{1:1} [k]}\sum_{B \subseteq  [k]} (-1)^{|B|} F_{[k]}(\w),
\end{equation}
where $\w =\w_{B,\pi} \in [0,1]^{k}$ is defined as
$$\w(t) = \left\{
\begin{array}{lcl}
a_{\pi(t)}  & \qquad &  \pi(t) \in B, \\
b_{\pi(t)} & & \pi(t) \in [k] \setminus B.
\end{array} \right.
$$
which by the symmetry of $F_{[k]}$ simplifies to the desired
$$
\alpha^{-k} \sum_{B \subseteq  [k]}  (-1)^{|B|} k! F_{[k]}(\a_B,\b_{[k] \setminus B}).
$$
\end{proof}

Note that in particular we have
\begin{equation}
\label{eq:mainsym}
\int_{[0,1]^k} \left(\partial_{\y_k,\z_k}  \cdots  \partial_{\y_1,\z_1}    F_{[k]}\right) da_1 \cdots da_k= \alpha^{-k} k! F_{[k]}(\b).
\end{equation}

Suppose for the sake of contradiction that the statement of the theorem is not true. Then there exists a largest  $k \in  [m] \setminus K(m,r,\alpha)$  such that $F_{[k]}$ is not zero almost everywhere. By \cref{eq:sym_derivZero} and \cref{eq:mainsym} we have
$$\alpha^{m-k} k!  \left(\sum_{i=0}^k {m-r \choose k-i}
{r \choose i}   \left(\frac{-\alpha}{1-\alpha}\right)^{i}\right)  F_{[k]}(\b) = 0,$$
for almost all $\b$, which then implies that $\sum_{i=0}^k {m-r \choose k-i}
{r \choose i}   \left(\frac{-\alpha}{1-\alpha}\right)^{i}=0$ as  $F_{[k]}$ is not zero almost everywhere. But this means that $k \in  K(m,r,\alpha)$, which is a contradiction.

\section{Concluding Remarks}

One of the main problems studied in the paper of Janson and S\'os~\cite{JansonSos} is determining  for which $(F,\alpha_1,\ldots,\alpha_m)$, the property $\mathcal{P}(F,\alpha_1,\ldots,\alpha_m)$  is always (i.e. for every $p \in (0,1]$) a \emph{quasi-random property}.  The only known  example for which this is \emph{not} the case is $\mathcal{P}(K_2,\frac{1}{2},\frac{1}{2})$. This fact was already observed by Chung and Graham in~\cite{MR1166604}. In the same paper, they also showed that $\mathcal{P}(K_2,\alpha,1-\alpha)$ is a quasi-random property for every $\alpha \in (0,1) \setminus \{\frac{1}{2}\}$.

\begin{conjecture}[{See {\cite[Conjecture 2.13 and Problem 2.16]{JansonSos}}}]
\label{conj:characterizeF}
Let $F \neq K_2$ be a non-empty graph with vertex set $\{1,\ldots,m\}$, and let $\alpha_1,\ldots,\alpha_m \in (0,1)$ satisfy $\sum_{i=1}^m \alpha_i \le 1$.  Then $\mathcal{P}(F,\alpha_1,\ldots,\alpha_m)$ is a quasi-random property for every $p \in (0,1]$.
\end{conjecture}

When $\sum_{i=1}^m \alpha_i<1$, \cref{conj:characterizeF} is verified in \cref{cor:quasiprop} (originally proved in \cite[Theorem 2.11]{JansonSos}). The case where $\alpha_1=\cdots=\alpha_m=\frac{1}{m}$ and $F$ is a regular graph, a star, or a disconnected graph  with at least one edge is verified by Janson and S\'os in \cite[Theorem 2.12]{JansonSos}.

Our \cref{thm:twins} settles the case when $\alpha_i$ are arbitrary and $F$ contains twin vertices. Prior to our work this was unknown even for the path on $3$ vertices and was stated as an open problem in \cite[Problem 2.19]{JansonSos}.

\section*{Acknowledgements}
The authors are grateful to Svante Janson for his invaluable comments and corrections to an earlier version of the paper. The second author would like to thank Alexander Razborov and Madhur Tulsiani for their helpful comments on the presentation of our main result.
%%%%%%%%%%%%%%%%%%%%%%%%%%%%%%%%%%%%%%%%%%%%%%%%%%%%%%%%%

\bibliographystyle{amsalpha}
\bibliography{HHLref}

\providecommand{\bysame}{\leavevmode\hbox to3em{\hrulefill}\thinspace}
\providecommand{\MR}{\relax\ifhmode\unskip\space\fi MR }
\providecommand{\MRhref}[2]{%
  \href{http://www.ams.org/mathscinet-getitem?mr=#1}{#2}
}
\providecommand{\href}[2]{#2}
\begin{thebibliography}{CGW89}

\bibitem[CG92]{MR1166604}
F.~R.~K. Chung and R.~L. Graham, \emph{Maximum cuts and quasirandom graphs},
  Random graphs, {V}ol.\ 2 ({P}ozna\'n, 1989), Wiley-Intersci. Publ., Wiley,
  New York, 1992, pp.~23--33. \MR{1166604 (93b:05149)}

\bibitem[CGW89]{MR1054011}
F.~R.~K. Chung, R.~L. Graham, and R.~M. Wilson, \emph{Quasi-random graphs},
  Combinatorica \textbf{9} (1989), no.~4, 345--362. \MR{1054011 (91e:05074)}

\bibitem[ES81]{MR615434}
B.~Efron and C.~Stein, \emph{The jackknife estimate of variance}, Ann. Statist.
  \textbf{9} (1981), no.~3, 586--596. \MR{615434 (82k:62074)}

\bibitem[Hoe48]{MR0026294}
Wassily Hoeffding, \emph{A class of statistics with asymptotically normal
  distribution}, Ann. Math. Statistics \textbf{19} (1948), 293--325.
  \MR{0026294 (10,134g)}

\bibitem[JS14]{JansonSos}
Svante Janson and Vera~T S{\'o}s, \emph{More on quasi-random graphs, subgraph
  counts and graph limits}, arXiv preprint arXiv:1405.6808 (2014).

\bibitem[LS06]{MR2274085}
L{\'a}szl{\'o} Lov{\'a}sz and Bal{\'a}zs Szegedy, \emph{Limits of dense graph
  sequences}, J. Combin. Theory Ser. B \textbf{96} (2006), no.~6, 933--957.
  \MR{2274085 (2007m:05132)}

\bibitem[Sha08]{MR2488748}
Asaf Shapira, \emph{Quasi-randomness and the distribution of copies of a fixed
  graph}, Combinatorica \textbf{28} (2008), no.~6, 735--745. \MR{2488748
  (2010b:05156)}

\bibitem[SS97]{MR1645698}
Mikl{\'o}s Simonovits and Vera~T. S{\'o}s, \emph{Hereditarily extended
  properties, quasi-random graphs and not necessarily induced subgraphs},
  Combinatorica \textbf{17} (1997), no.~4, 577--596. \MR{1645698 (99j:05115)}

\bibitem[SY10]{MR2591049}
Asaf Shapira and Raphael Yuster, \emph{The effect of induced subgraphs on
  quasi-randomness}, Random Structures Algorithms \textbf{36} (2010), no.~1,
  90--109. \MR{2591049 (2011a:05320)}

\bibitem[Tho87a]{MR930498}
Andrew Thomason, \emph{Pseudorandom graphs}, Random graphs '85 ({P}ozna\'n,
  1985), North-Holland Math. Stud., vol. 144, North-Holland, Amsterdam, 1987,
  pp.~307--331. \MR{930498 (89d:05158)}

\bibitem[Tho87b]{MR905280}
\bysame, \emph{Random graphs, strongly regular graphs and pseudorandom graphs},
  Surveys in combinatorics 1987 ({N}ew {C}ross, 1987), London Math. Soc.
  Lecture Note Ser., vol. 123, Cambridge Univ. Press, Cambridge, 1987,
  pp.~173--195. \MR{905280 (88m:05072)}

\bibitem[Yus10]{MR2676838}
Raphael Yuster, \emph{Quasi-randomness is determined by the distribution of
  copies of a fixed graph in equicardinal large sets}, Combinatorica
  \textbf{30} (2010), no.~2, 239--246. \MR{2676838 (2011j:05302)}

\end{thebibliography}
\end{document}